\documentclass[reqno,a4paper,12pt]{amsart} 

\usepackage{amsmath,amscd,amsfonts,amssymb}
\usepackage{mathrsfs,dsfont}
\usepackage{color}
\usepackage{mathtools}
\usepackage{hyperref}
\usepackage{enumitem}
\usepackage{todonotes}
\usepackage[normalem]{ulem}
\usepackage{subcaption}

\usepackage{tikz}
\usetikzlibrary{decorations.pathreplacing}

\numberwithin{equation}{section}
\numberwithin{figure}{section}

\def\R{\mathbb{R}}

\def\Q{\mathbb{Q}}
\def\Z{\mathbb{Z}}
\def\N{\mathbb{N}}

\def\1{\mathds{1}}

\newcommand{\Tile}{\operatorname{Tile}}

\renewcommand\le{\leqslant}

\renewcommand\leq{\leqslant}
\renewcommand\geq{\geqslant}

\theoremstyle{plain}
\newtheorem{thm}{Theorem}[section]
\newtheorem{theorem}{Theorem}[section]

\newtheorem{lemma}[thm]{Lemma}
\newtheorem{corollary}[thm]{Corollary}

\newtheorem*{claim*}{Claim}

\theoremstyle{definition}
\newtheorem{definition}[thm]{Definition}
\newtheorem*{definition*}{Definition}
\newtheorem*{remarks*}{Remarks}
\newtheorem*{remark*}{Remark}
\newtheorem{remark}[thm]{Remark}

\newcommand{\quake}[2][\vec v]{\operatorname{Plates}(#2,#1)}

\begin{document}

\title[Translational tilings by rational polygonal sets]{Periodicity and decidability of translational tilings by rational polygonal sets}

\author[J. de Dios Pont]{Jaume de Dios Pont}
\address{Department of Mathematics, ETHZ, 8001 Zurich.}
\email{jaumed@math.ethz.ch}
     
\author[J. Greb\'ik]{Jan Greb\'ik}
\address{Faculty of Informatics, Masaryk University, Botanick\'{a} 68A, 602 00 Brno, Czech Republic and UCLA Mathematics, 520 Portola Plaza, Los Angeles, CA 90095, USA.}
\email{grebikj@gmail.com}
     
\author[R. Greenfeld]{Rachel Greenfeld}
\address{Department of Mathematics, Northwestern University, Evanston, IL 60208.}
\email{rgreenfeld@northwestern.edu}

\author[J. Madrid]{Jos\'e Madrid}
\address{Department of  Mathematics, Virginia Polytechnic Institute and State University,  225 Stanger Street, Blacksburg, VA 24061-1026, USA.}
\email{josemadrid@vt.edu}

    	\subjclass[03B25, 52C22, 52C23]{03B25, 52C22, 52C23}
    	\date{}
    	

\begin{abstract} 
The periodic tiling conjecture asserts that if a region $\Sigma\subset \mathbb R^d$ tiles $\mathbb R^d$ by translations then  it admits at least one \emph{fully periodic} tiling. This conjecture is known to hold in $\mathbb R$, and recently it was disproved in sufficiently high dimensions. 
In this paper, we study the periodic tiling conjecture for \emph{polygonal sets}: bounded open sets in $\mathbb R^2$ whose boundary is a finite union of line segments. We prove the periodic tiling conjecture for any polygonal tile whose vertices are rational. As a corollary of our argument, we also obtain the decidability of tilings by rational polygonal sets. 
Moreover, we prove that \emph{any} translational tiling by a rational polygonal tile is weakly-periodic, i.e., can be  partitioned into finitely many singly-periodic pieces.  
\end{abstract}

\maketitle


\section{Introduction}

Let $\Sigma \subset \R^d$ be a bounded, measurable set. We say that $\Sigma$ \emph{tiles $\R^d$ by translations} if there exists a countable set $T\subset \R^d$ such that the translations of $\Sigma$ along the points of $T$, $\Sigma + t$, $t\in T$, cover almost every point in $\R^d$ exactly once. In this case, we write $\Sigma \oplus T=\R^d$ and refer to $\Sigma$ as a \emph{translational tile} and to $T$ as a \emph{tiling of $\R^d$ by $\Sigma$}. 
Let $\Tile(\Sigma; \R^d)$ denote the set of all the tilings of $\R^d$ by $\Sigma$, i.e., $$\Tile(\Sigma; \R^d)=\{T\subset \R^d\colon \Sigma\oplus T=\R^d\}.$$
Since in this work we consider tilings by \emph{only} translations, in what follows we sometime abbreviate and write ``tiling'' to refer to  tiling by translations and ``tile'' to refer to  translational tile.

The study of translational tiling consists of studying the structure of sets in $\Tile(\Sigma; \R^d)$. A major conjecture in this area is \emph{the periodic tiling conjecture} \cite{stein, grunbaum-shephard, LW}, which asserts that if $\Tile(\Sigma; \R^d)$ is non-empty, then it must contain at least one \emph{periodic} set, i.e., a set $T\in \Tile(\Sigma; \R^d)$, which is invariant under translations by  lattice $\Lambda\subset \R^d$ (here, a lattice is a full rank discrete subgroup).

The periodic tiling conjecture (PTC) is known to hold in $\R$ \cite{LW}, for convex domains in all dimensions \cite{V, M} and for topological disks in $\R^2$ \cite{gbn, ken, err}. On the other hand, recently, the periodic tiling conjecture was disproved in sufficiently high dimensions \cite{GT22}, even under the assumption that the tile is connected \cite{GK}. 

One important motivation to study the structure of tilings and the periodic tiling conjecture in particular is the connection to the \emph{decidability} of tilings. Indeed, in \cite{wang} it was shown that if the periodic tiling conjecture were true, then  the question of whether a set is a tile or not would be decidable.
Recently, in \cite{GT23}, it was shown that translational tilings are \emph{undecidable} if the dimension is unbounded. However, the decidability of translational tilings (by a single tile) in $\R^d$, for a fixed $d\geq 2$, is still open. In this paper we study the periodicity and decidability of planar translational tilings by \emph{polygonal sets}.

\subsection{Polygonal tiles} 
A \emph{polygonal set} is a bounded open set $\Omega \subset \mathbb R^2$, whose boundary is a finite union of segments. Note that a polygonal set can certainly be disconnected and its connected components are not necessarily simply-connected.  We denote the  \emph{vertices} of $\Omega$ by $V(\Omega)$, and the \emph{edges} of  $\Omega$ by $E(\Omega)$. See Figure  \ref{fig:poly2} for illustration.

\begin{figure}[!h]
    \includegraphics[width = .8\textwidth]{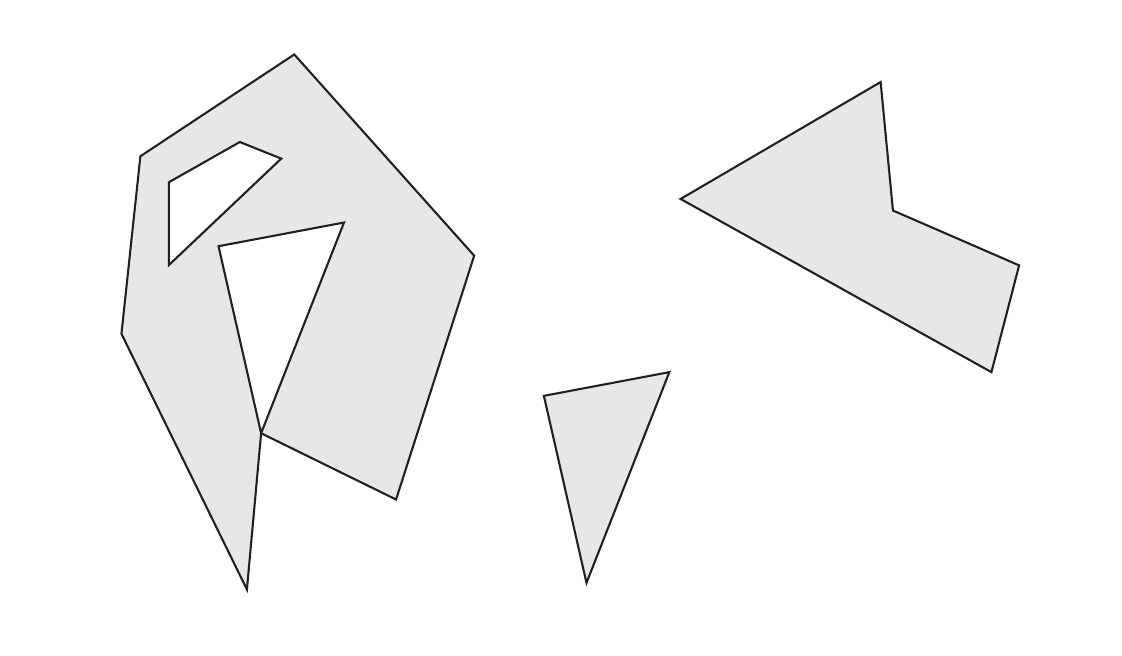}
    \caption{Example of a polygonal set $\Omega$. The set $\Omega$ is shown in gray; it has three connected components. The set $E(\Omega)$ consists of the $22$ segments in black whose union is the boundary of $\Omega$, $\partial \Omega$, and $V(\Omega)$ is the $21$ endpoints of those segments.}
    \label{fig:poly2}
\end{figure}

We say that a polygonal set $\Omega$ is \emph{rational} if  the set $V(\Omega)-v$ is contained in $\Q^2$, for any $v\in V(\Omega)$.

Our first result is that the periodic tiling conjecture holds for rational polygonal tiles.

\begin{theorem}[PTC holds for rational polygonal tiles]\label{thm:ptc}
    A rational polygonal set $\Omega \subset \mathbb R^2$  tiles $\mathbb R^2$ by translations if and only if $\Tile(\Omega; \R^2)$ contains a periodic set. 
\end{theorem}

Using the proof of Theorem~\ref{thm:ptc} we also obtain the decidability of tilings in $\R^2$ by rational polygonal tiles.

\begin{corollary}\label{cor:decide}
    There is an algorithm that computes, upon any given rational polygonal set $\Omega$, whether $\Tile(\Omega;\R^2)$ is empty or not.
\end{corollary}

In addition, we establish a structural result that applies to \emph{any} set in $\Tile(\Omega; \R^2)$. To state our result, we first introduce some refinements of the notion of periodicity.

\begin{definition}[Single-periodicity, weak-periodicity, double-periodicity] Let $S\subset \R^2$ be a countable set.
\begin{itemize}
    \item[i)] We say that $S$ is  \emph{singly-periodic} (or $\langle h\rangle$-periodic) if it is invariant under translations by some non-zero vector $h \in \R^2\setminus \{0\}$.
    \item[ii)] We say that $S$ is \emph{weakly-periodic}  if it can be partitioned into finitely many  singly-periodic sets.
    \item[iii)] We say that $S$ is \emph{double-periodic} (or $\langle h_1,h_2\rangle$-periodic) if it is periodic, i.e., invariant under translations by two linearly independent vectors $h_1,h_2 \in \R^2$.
\end{itemize}
\end{definition}

We show that any tiling by a  rational polygonal tile must be weakly-periodic.

\begin{theorem}[Structure of tilings by rational polygonal sets]\label{thm:weakper}
    Let $\Omega \subset \mathbb R^2$ be a rational polygonal set. Then any $T\in \Tile(\Omega;\R^2)$  is weakly-periodic. 
\end{theorem}

We will address the periodicity and decidability of polygonal sets with irrational vertices in a subsequent work.

\subsection{Outline of our arguments}
Let $\Omega \subset \mathbb R^2$ be a rational polygonal set. By translation and dilation invariance of tilings, we can assume that $V(\Omega)\subset \Z^2$. 

\begin{itemize}
    \item[(1)] In Section~\ref{sec:integer}, by sliding the \emph{sliding components} of a tiling by $\Omega$, we show that there exists a tiling $T$ of $\R^2$ by $\Omega$ that is contained in $\Z^2$.
    \item[(2)]  In Section~\ref{sec:discrete}, we \emph{discretize} $\Omega$: We show that there is a natural number $N$ and a finite set $F\subset N^{-1}\Z^2$ such that every $0\in T\subset \R^2$ is a tiling of $N^{-1}\Z^2$ by $F$ if and only if it is an integer tiling of $\R^2$ by $\Omega$.
    \item[(3)] Then, using \cite{BH, GT}, where the periodic tiling conjecture was proved to hold in $\Z^2$, we show the existence of a periodic $T\subset \Z^2$ which is a tiling of $N^{-1}\Z^2$ by $F$. Theorem~\ref{thm:ptc} and Corollary~\ref{cor:decide} then follow from combining this with (1) and (2). 
    \item[(4)]  While the proof of Theorem~\ref{thm:ptc} uses  off-the-shelf results from \cite{BH, GT}, the proof of Theorem~\ref{thm:weakper} is more delicate and requires novel refinement of the technology introduced in \cite{GT}. 
    Indeed, by combining (2) and \cite[Theorem 1.4]{GT}, we merely obtain that any \emph{integer} tiling in $\Tile(\Omega;\R^2)$ is weakly periodic.
     The main obstacle to transferring Theorem~\ref{thm:weakper} from its integer counterpart is the possible \emph{sliding} components obstructing periodicity.  
  In Section~\ref{sec:weakper}, we address this obstacle by analysing the connection between the sliding components of \emph{any} set in $\Tile(\Omega; \R^2)$ and the singly-periodic components of sets in $\Tile(F; N^{-1}\Z^2)$, and  conclude Theorem~\ref{thm:weakper}.  
\end{itemize}

\subsection{Notation}

Let $G=(G,+)$ be an Abelian group.
For $A\subset G$, $g\in G$ we use the notation $A+ g$ for the set $\{a + g\colon a\in A\}$; the set $A-g$ is defined similarly. For $A, B\subset G$, we write $$A+B\coloneqq \{a+b\colon a\in A; b\in B\};$$ the set $A-B$ is defined similarly.

For $\Sigma\subset \R^d$, we denote the  closure of $\Sigma$ by $\overline{\Sigma}$ and the boundary of $\Sigma$ by $\partial \Sigma$.

\subsection{Acknowledgements}
We are grateful to Terence Tao and the anonymous referee for their helpful suggestions, which improved the exposition of this paper.

JD was partially supported by a UCLA Dissertation Year Fellowship award and the Simons Collaborations in MPS grant 563916.  JG was supported by MSCA Postdoctoral Fellowships 2022 HORIZON-MSCA-2022-PF-01-01 project BORCA grant agreement number 101105722.
RG was supported by the Association of Members of the Institute for Advanced Study (AMIAS) and by NSF grant DMS-2242871. JM was supported by the AMS Stefan Bergman Fellowship.

\section{From a tiling to an integer tiling}\label{sec:integer}
Let $\Omega \subset \mathbb R^2$ be a rational polygonal set. As tilings and periodicity are both invariant under affine transformations, by translating $\Omega$ to make the origin one of its vertices and  dilating by the least common multiple of the denominators of both coordinated of all the vertices  we can assume without loss of generality that the vertices of $\Omega$ are in $\mathbb Z^2$. In other words, we may assume that $\Omega$ is an \emph{integer} set. If  $\Omega$ tiles we say that it is an \emph{integer tile}.

\begin{definition}[Integer tiling]
    A tiling $T\in\Tile(\Omega;\R^2)$ is called an \emph{integer tiling}  if $V(\Omega)+T$ is entirely contained in $\Z^2$.
\end{definition}

The goal of this section is showing that any tiling of $\R^2$ by  $\Omega$ gives rise to an integer tiling $T'$ of $\R^2$ by $\Omega$, as stated in the following lemma.

\begin{lemma}[Existence of integer tiling]\label{lem:integertiling}
    Let $\Omega\subset \R^2$ be an integer polygonal tile. Then there exists  $T \in\Tile(\Omega ; \R^2)$ such that $T\subset \Z^2$.
\end{lemma}

We will use the vertex-adjacent equivalence relationship:

\begin{definition}[Integer tile, vertex-sharing equivalence class]
Let $\Omega$ be an integer  tile. 
For $t, t'$ in $\mathbb R^2$ we say that $t\approx_\Omega t'$ if $(V(\Omega)+t)\cap (V(\Omega)+t')$ is not empty. 
If $T$ is a tiling of $\mathbb R^2$ by $\Omega$, the relationship $\approx_\Omega$ induces (by transitive closure) an equivalence relationship in $T$. We denote this equivalence relationship by $t\sim_{\Omega} t'$.
\end{definition}

\subsection{Reduction}

The following lemma shows that to prove  Lemma~\ref{lem:integertiling} it suffices to show that every polygonal tile can be repaired to a tile with a unique vertex sharing component.

\begin{lemma}[Vertex-sharing tiling by an integer tile is an integer tiling]
\label{lem:tile+vertex=>tiling}
Let $\Omega$ be an integer tile, so that $T$ is a tiling of  $\mathbb R^2$ by $\Omega$ containing $0$. If $t\sim_{\Omega} t'$ for any elements $t,t'$ of $T$, then $T\subset \mathbb Z^2$.
\end{lemma}
\begin{proof}
If $\Omega$ is an integer tile $t\approx_\Omega t'$ implies that $t-t'\in \mathbb Z^2$.  Let $t \in T$. By assumption, since $0\in T$, we have $0 \sim_{\Omega} t$. This means that there is a path $0 = t_1, t_2, \dots, t_n = t$ with $t_j \in T$ and $t_i \approx_\Omega t_{i+1}$ for $i=1,\dots,n-1$. Therefore,
$$
t - 0 =\sum_{j=1}^{n-1} t_{i+1}-t_i
$$ is in $\Z^2$, as claimed. 
\end{proof}

\subsection{Sliding} 
Lemma~\ref{lem:tile+vertex=>tiling} reduces the proof of Lemma~\ref{lem:integertiling} to showing, for any integer tile $\Omega\subset \R^2$, the existence of a tiling $T\in \Tile(\Omega;\R^2)$ that is \emph{$\sim_{\Omega}$-connected}, i.e., has a single $\sim_{\Omega}$-equivalence class.
 This will be achieved in this subsection, by developing a \emph{sliding} machinery, which allows one to shift the $\sim_{\Omega}$-components of a given tiling $T\in \Tile(\Omega;\R^2)$, while preserving the tiling, to eventually merge all of the components into a tiling $T'\in \Tile(\Omega; \R^2)$ with a \emph{single} $\sim_{\Omega}$-equivalence class (see Figure~\ref{fig:slide} for an illustration).

This is done by first analysing the structure of the $\sim_{\Omega}$-tiling-components, and showing that they are invariant under \emph{sliding} in some direction $\vec v\in \R^2$, i.e., $\Omega+T_i +\delta \vec v=\Omega+T_i$, for every $\delta\in \R$ and $T_i\in T/\sim_\Omega$ (see Figure~\ref{fig:components} for an illustration).

\begin{lemma} [Tiling-components slide]
\label{lem:tilings_slide}
Let $T$ be a tiling of $\mathbb R^2$ by a polygonal tile $\Omega$, let $T_i$, $i\in I$, be the  equivalence classes of \begin{equation}\label{eq:tiling_slide}
    T=\bigsqcup_{i\in I} T_i
\end{equation}
by  $\sim_{\Omega}$. Then there exists a unit vector $\vec v\in \R^2$ such that each set $W_i \coloneqq \Omega + T_i$ is $\vec v   \mathbb R-$invariant. 
\end{lemma}

\begin{figure}[!h]
    \centering
    \includegraphics[width=0.5\linewidth]{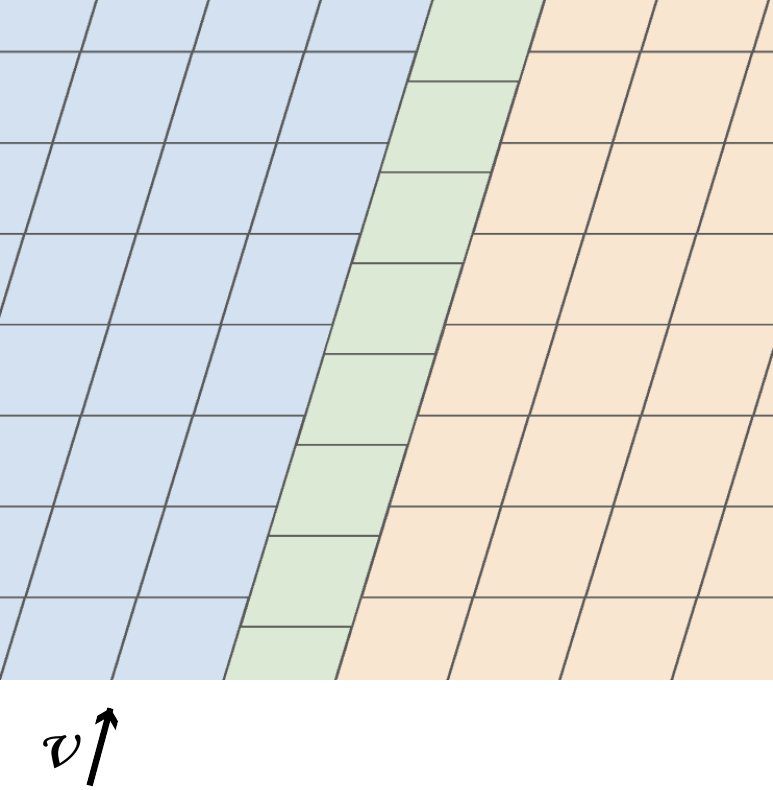}
    \caption{In blue, green and orange, three \emph{sliding components}, $W_1, W_2, W_3$, of a tiling, with a sliding direction $v$.}
    \label{fig:components}
\end{figure}

\begin{remark}
    See also \cite[Theorem 1.7]{ggrt} for a related decomposition into ``rational componenets'' of tilings of a torus.
\end{remark}

\begin{proof}
For  $ D\subset \mathbb R^d$ we say that $x\in \partial D$ is a \emph{proper vertex of $D$} if for no neighborhood $B_\epsilon(x)$, the intersection of $B_\epsilon(x)\cap \partial D$ is a segment.

We first show that if $x_0\in \partial(\Omega+S)$ is a proper vertex of $\Omega+S$, then $S$ is not $\sim_{\Omega}$-closed, i.e., there exist $t\in T\setminus S$ and $s\in S$ such that $s\approx_\Omega t$.

Indeed, for all $\epsilon$ there exists a $t_\epsilon$ in $T\setminus S$ so that $B_\epsilon(x_0)\cap (\Omega+t_{\epsilon}) \neq \emptyset$. Since the set $T$ is uniformly discrete and $\Omega$ is bounded, there must exist $t\in T\setminus S$ so that $(\Omega + t) \cap B_{\epsilon}(x_0) \neq \emptyset$ for all $\epsilon>0$ (or, equivalently, $x_0\in \overline{\Omega} + t$). 
    Let $\{t_1, \dots t_n \}$ be the set of all elements in $T\setminus S$ with the property that $x_0 \in \overline{\Omega} +t_i$, and $\{s_1, \dots s_m\}$ be the set of the elements in $S$ with the property that $x_0 \in \overline{\Omega} +s_i$. (Note that none of these sets is empty.) Then we must be in one of the following three scenarios:
    \begin{enumerate}
        \item $m=1$, and $x_0\not\in V(\Omega+ s_1)$;
        \item $n=1$, and $x_0\not\in V(\Omega + t_1)$;
        \item There are $1\leq j\leq n$ and $1\leq k\leq m$ such that $x_0\in V(\Omega + t_j)\cap V(\Omega + s_k)$.
    \end{enumerate}

    In the first case, we have that $x_0\in \overline{\Omega}+s$ for $s\in S$ only if $s=s_1$.
    Therefore $x_0$ is not a proper vertex of $\Omega+S$ as $x_0\not\in V(\Omega)+s_1$.    
    In the second case, by the same argument we show that $x_0$ is not a proper vertex of $\Omega + (T\setminus S)$.
    This leads to a contradiction, since $\partial(\Omega + (T\setminus S))= \partial(\Omega + S)$. In the third case, our conclusion holds; indeed, in this case, $t_j\approx_\Omega s_k$.

Now, if an open set $D$ has no proper vertices, then $\partial D$ is a closed subset of $\mathbb R^2$ which is locally a line, and therefore a union of parallel lines in  direction $\vec v\in \R^2$. Since $\partial D$ is invariant under translations by $\vec v$, also $D$ is invariant under translations by $\vec l$. Indeed,  since $\partial D$ does not contain $x$ for any point $x\in D$, we must have that  $\partial D$ does not intersect $x+\vec v\mathbb R$. In particular, $x+\vec v\mathbb R\subset D$.
 
We conclude that  $\mathbb R^2 = \bigsqcup_{i\in I} \Omega + T_i$, where $\Omega + T_i$ is $\vec v_i$ invariant. However, as $\Omega + T_i$ is not empty, and $\Omega + T_i\cap \Omega + T_j = \emptyset$ if $i\neq j$, all of the vectors $\vec v_i$ must be parallel. This concludes the proof. 
\end{proof}

Using Lemma~\ref{lem:tilings_slide}, one can now slide the $\sim_\Omega$-components of a given tiling $T\in\Tile(\Omega;\R^2)$ to merge them into a single $\sim_\Omega$-component of another tiling $T'\in Tile(\Omega;\R^2)$.

\begin{lemma}[Merge components by sliding] \label{lem:onecomp}
Suppose that $T$ is a tiling of $\R^2$ by a polygonal tile $\Omega$, and let $T_1, T_2,\dots$ and $\vec v\in\R^2\setminus \{0\}$ be as in Lemma~ \ref{lem:tilings_slide}. Then there exist real coefficients $s_1, s_2,\dots$ such  that $$T'\coloneqq \bigsqcup T_i+s_i \vec v\in \Tile(\Omega;\R^2)$$  and $T'$ has a single $\sim_{\Omega}$-equivalence class.
\end{lemma}

\begin{figure}[!h]
    \centering
    \includegraphics[width=0.4\linewidth]{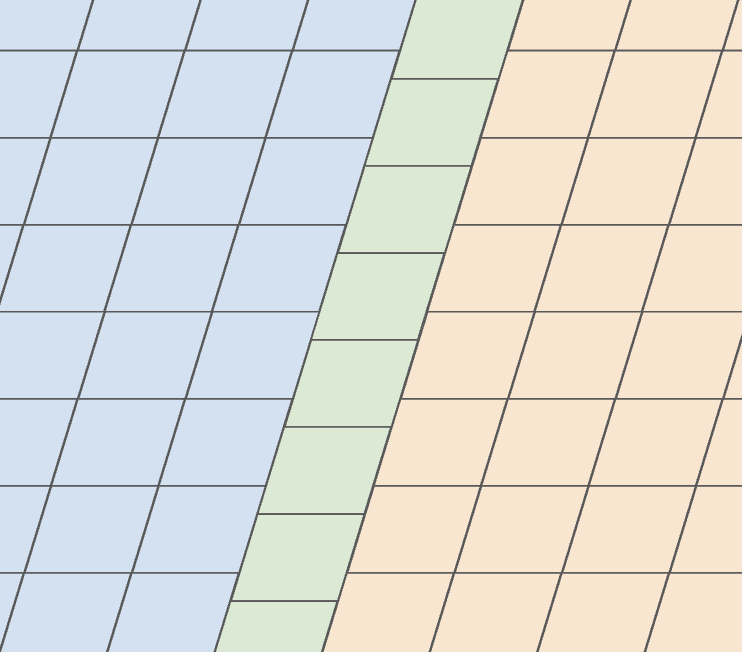}
    \hspace{2em}
    \includegraphics[width=0.415\linewidth]{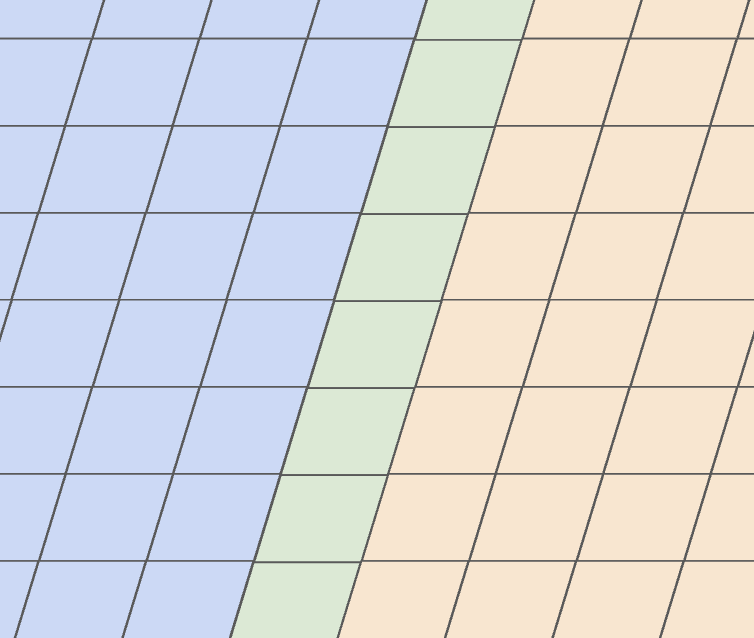}
    \caption{The  $\sim_\Omega$-classes from Figure~\ref{fig:components} before sliding (left) and after sliding (right).}
    \label{fig:slide}
\end{figure}

\begin{proof}
    Applying an affine transformation, we may assume without loss of generality that $v=(0,1)$. Denote $W_i\coloneqq \Omega + T_i$.
    Let $L= \R\times \{0\}$, and $L_i\coloneqq  \{(x,0)\colon \exists y, (x,y)\in W_i \}$. Then, since $\Omega$ is open and $T=\bigsqcup T_i$ is a tiling by $\Omega$, the sets $L_i$ are open, non-overlapping and non-empty.   Since the set $T$ is a tiling by $\Omega$, the distance between any two elements in $T$ is at least  $$\delta\coloneqq \min\{|x|\colon x\in \R^2,\; \Omega\cap(\Omega+x) \text{ has zero measure }\}>0.$$
    Thus, for any $R$ there are finitely many $L_i$ at distance at most $R$ from $0$, and so, by reordering if needed, we may assume that the sets $L_i$ are ordered so that $d(L_{i},0)\le d(L_{i+1},0)$.
    This implies that for $i>1$, one has $\overline L_i \cap  \bigcup_{j< i} \overline L_j\neq \emptyset$. 
    
    We will define the numbers $s_i$ by induction.  Set $s_1=0$. For $i>1$ suppose that  $s_1, \dots , s_{i-1}$ were already chosen to satisfy the claim. We will now define $s_i$  from the values of $s_1, \dots , s_{i-1}$. From the ordering of $L_1,L_2,\dots$ it follows that there must exist $j<i$ and $x\in L$ such that $x\in \partial L_i \cap \partial L_j$, or equivalently, $$x+v\R \subset \partial (\overline \Omega + T_i)\cap \partial (\overline \Omega + T_j).$$ Thus, there are $t_i\in T_i, t_j \in T_j$ such that both $V(\Omega+t_i)\cap (x+\vec v\R)$ and $V(\Omega+t_j)\cap (x+\vec v\R)$ are non-empty. We can therefore choose  $v_i\in V(\Omega+t_j)\cap (x+\vec v\R)$ and $v_j \in V(\Omega+t_j)\cap (x+\vec v\R)$. We then set  $s_i\in \R$ to satisfy $$(s_j-s_i) \vec v= v_i - v_j,$$ which ensures that an element of $T_i+s_i\vec v$ is $\sim_{\Omega}$-connected to an element of $T_j+s_j\vec v$, while preserving the $\sim_{\Omega}$-connectedness of $T_i$ and $T_j$.
\end{proof}

Combining Lemmas \ref{lem:tile+vertex=>tiling} and \ref{lem:onecomp} we obtain Lemma~\ref{lem:integertiling}, as needed.

\section{Encoding an integer tile as a discrete tile}\label{sec:discrete}
In this section, we prove Theorems~\ref{thm:ptc} and Corollary~\ref{cor:decide} by a reduction from the continuous setup to the discrete setup.

\subsection{Discrete translational tilings} 
Let $(G,+)$ be a finitely generated Abelian group. A finite set $F\subset G$ tiles $G$ by translations if there exists a set $T\subset G$ such that $F\oplus T=G$, namely, the translations of $F$ along $T$: $F+t$, $t\in T$, cover every point in $G$ exactly once. As in the continuous setup, we denote
$$\Tile(F;G)\coloneqq \{T\subset G\colon F\oplus T=G\}.$$ 
The discrete version of the periodic tiling conjecture is know to hold in $\Z$ \cite{N} and $\Z^2$ \cite{BH, GT}, but was recently disproved in $\Z^2\times G_0$ for some high-rank finite groups $G_0$, and thus (using \cite{bgu}) in $\Z^d$ for sufficiently large $d$ \cite{GT22}. 

\subsection{Discretization} 
In this subsection, we encode any rational polygonal set $\Omega$ as a finite set in a two-dimensional lattice, in a way that preserves the tiling properties of $\Omega$. This will enable us to utilize the proof of the two-dimensional discrete periodic tiling conjecture \cite{BH, GT} to obtain Theorem~\ref{thm:ptc} and deduce Corollary~\ref{cor:decide}.

Let $\Omega$ be an integer  set.
    Consider all integer translates of the finitely many segments in $E(\Omega)$. These segments partition  $[0,1]^2$ into $M$ open polygonal sets $P_0,\dots,P_{M-1}$ (up to null sets), see Figure \ref{fig:lines}.  
    \begin{figure}[ht]
    \centering
    
    \includegraphics[width = .35\textwidth]{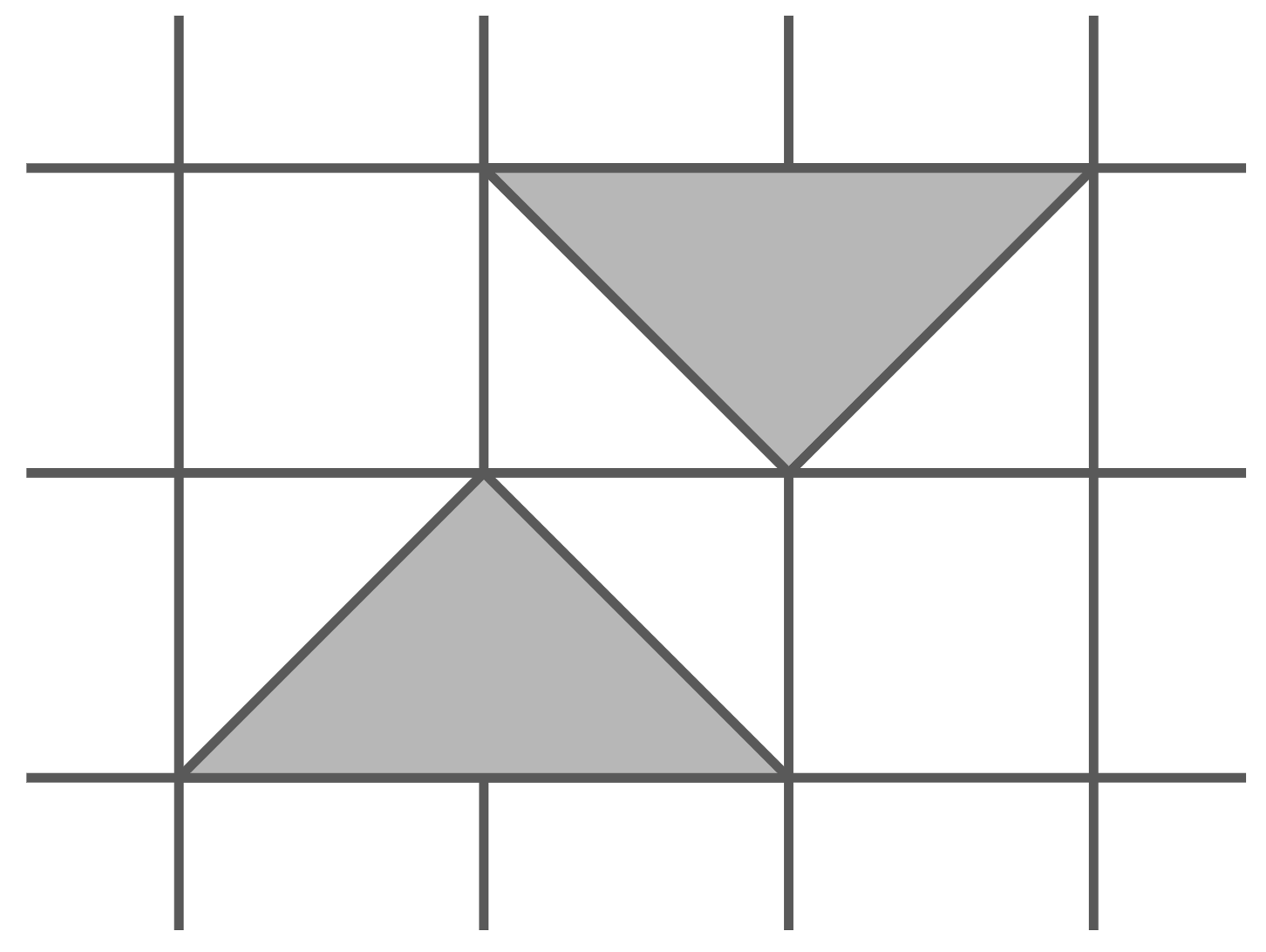}
    \hspace{.15\textwidth}
    \includegraphics[width = .35\textwidth]{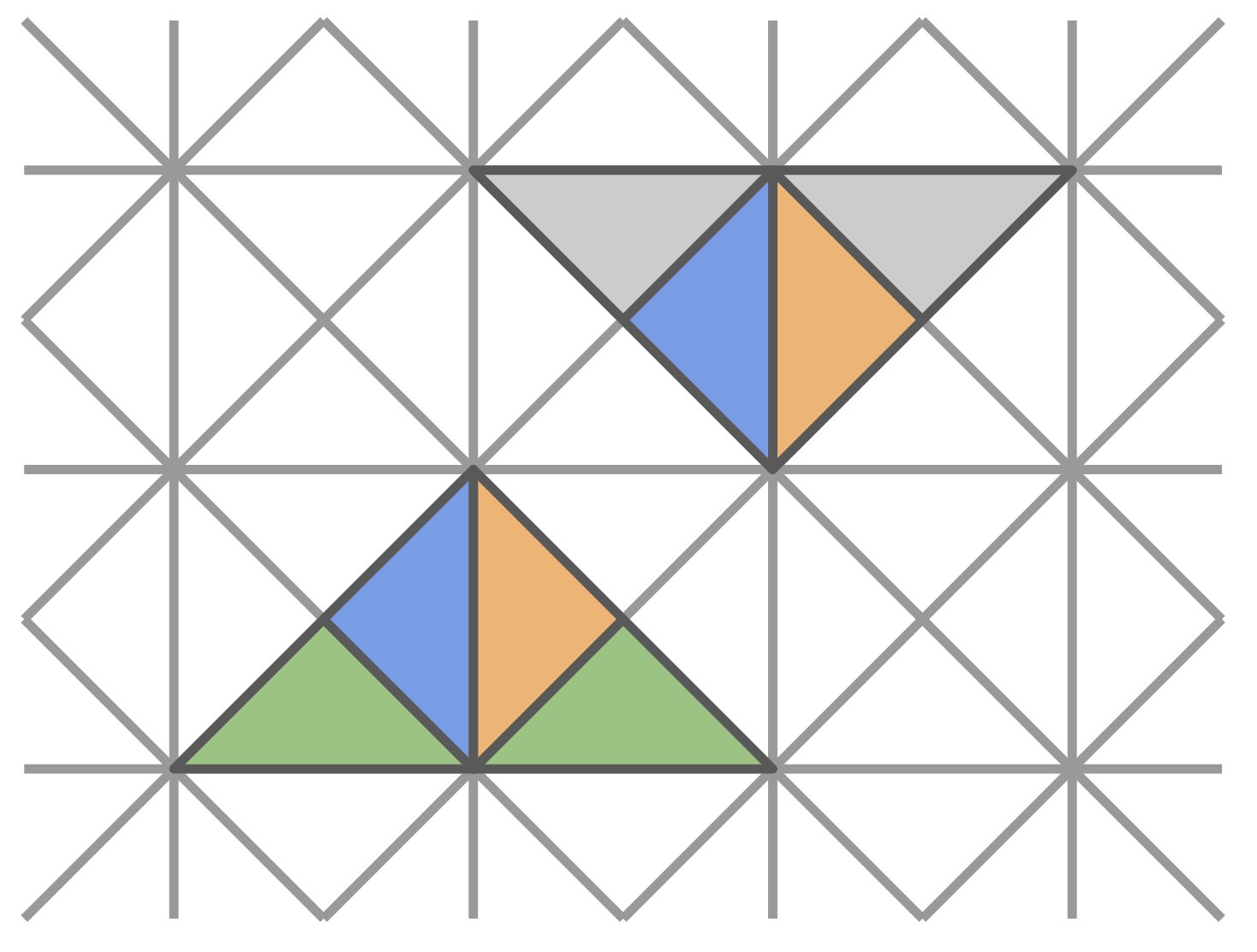}

    \caption{(Left) An integer set $\Omega$ and (Right) its associated partition into translations of subsets of $[0,1]^2$, with $P_0$ in grey; $P_1$ in green; $P_2$ in orange; and $P_3$ in blue.}
    \label{fig:lines}
\end{figure}

    Observe that for each $v \in \Z^2$ and each $P_i$, either $(v + P_i)$ is in $\Omega$ or is disjoint from $\Omega$. Let $N = N(\Omega) = 3k+4$, where $k\in\N$ is such that $M \leq k^2$. To each $P_i$ we associate a subset $S_i$ of $\{0,\dots,N-1\}^2$ as follows:
    \begin{itemize}
     \item[1.] For $1\leq  i\leq M-1$, let $(a_i,b_i)\in \{1,\dots,k \}^2$ be the unique pair such that $(a_i-1)+k(b_i-1)=i$.
     We define the set $S_i$ to consist of the $8$ points $$(\tilde a, \tilde b)\in \{0,\dots,N-1\}^2\setminus \{(3a_i,3b_i)\}$$ which satisfy $|\tilde a- 3a_i|\le 1$ and $|\tilde b - 3b_i|\le 1$.
     \item[2.] For $i=0$, we first define $\tilde S_0 \coloneqq \{0,\dots,N-1\}^2  \setminus \bigcup_{j=1}^{M-1} S_j$. The set $S_0$ is defined as $\tilde S_0 \cup \{(N,1),(1,N)\} \setminus\{(0,1), (1,0)\}$.
    \end{itemize}
    See Figure~\ref{fig:Si}.
    \begin{figure}[ht]%
    \centering
     \subfloat[\centering When $M=4$, $k=2$ and $N=10$ (e.g., for $\Omega$ as in Figure \ref{fig:lines}).]{{\includegraphics[width=6cm]{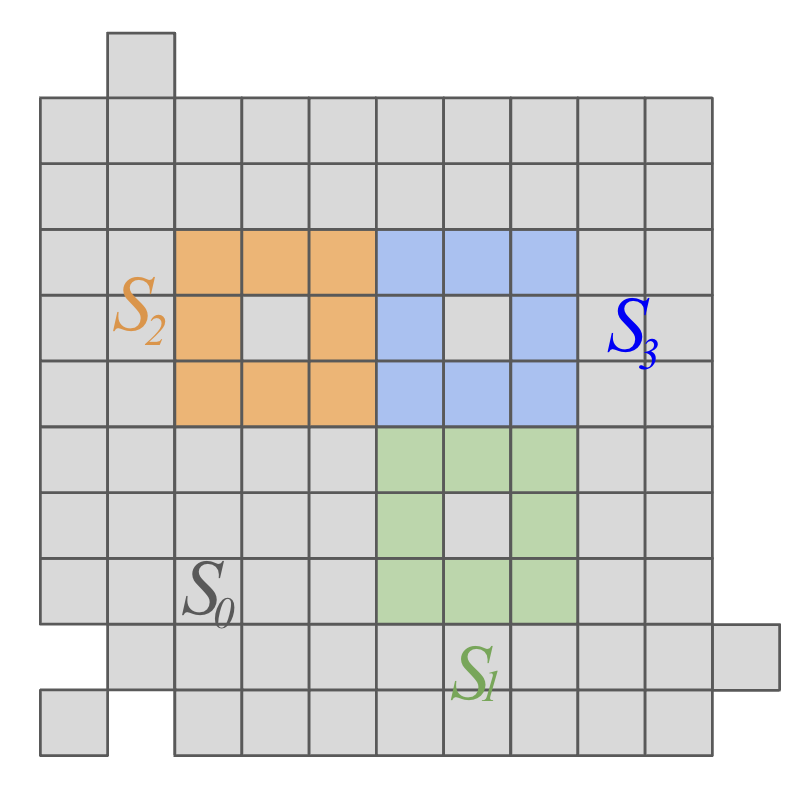} }}%
    \qquad
    \subfloat[\centering When $M=7$, $k=3$ and $N=13$.]{{\includegraphics[width=6cm]{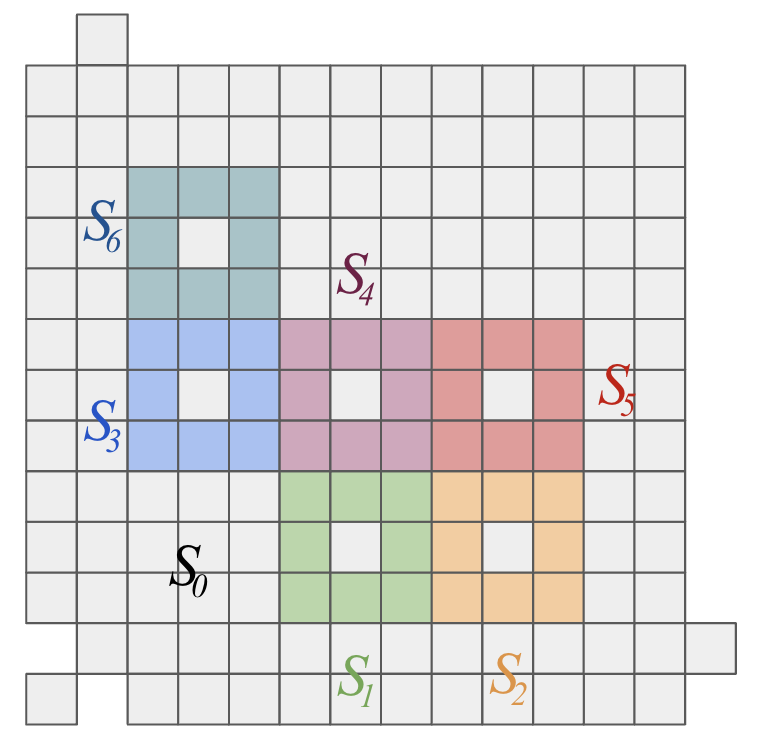} }}%
    \caption{The associated sets $S_0,\dots,S_{M-1}$ in two different cases. (Here, a point $(a,b)\in\{0,\dots,N\}^2$ is represented by a square $(a,b)+[0,1]^2$.)}
    \label{fig:Si}%
\end{figure}

 Now we associate to $\Omega$ the (discrete) set:
\begin{equation}\label{eq:F}
F = F(\Omega) \coloneqq \bigsqcup_{v \in \mathbb Z^2}\; \bigsqcup_{\substack{i=0,\dots ,M-1\\ (v + P_i) \subset \Omega}} \left( v+ \tfrac 1 N S_i \right)\subset \tfrac 1 N  \Z^2.
 \end{equation} 
 (See Figure \ref{fig:new_tile}.) Thus, for every $v\in\Z^2$, $S_i\in\{S_0,\dots,S_{M-1}\}$, either $(v+ \frac{1}{N} S_i)\subset F$ or $(v+\frac{1}{N} S_i)\cap F$ is empty.
 
    \begin{figure}[ht]%
    \centering
    \includegraphics[width=.8\textwidth]{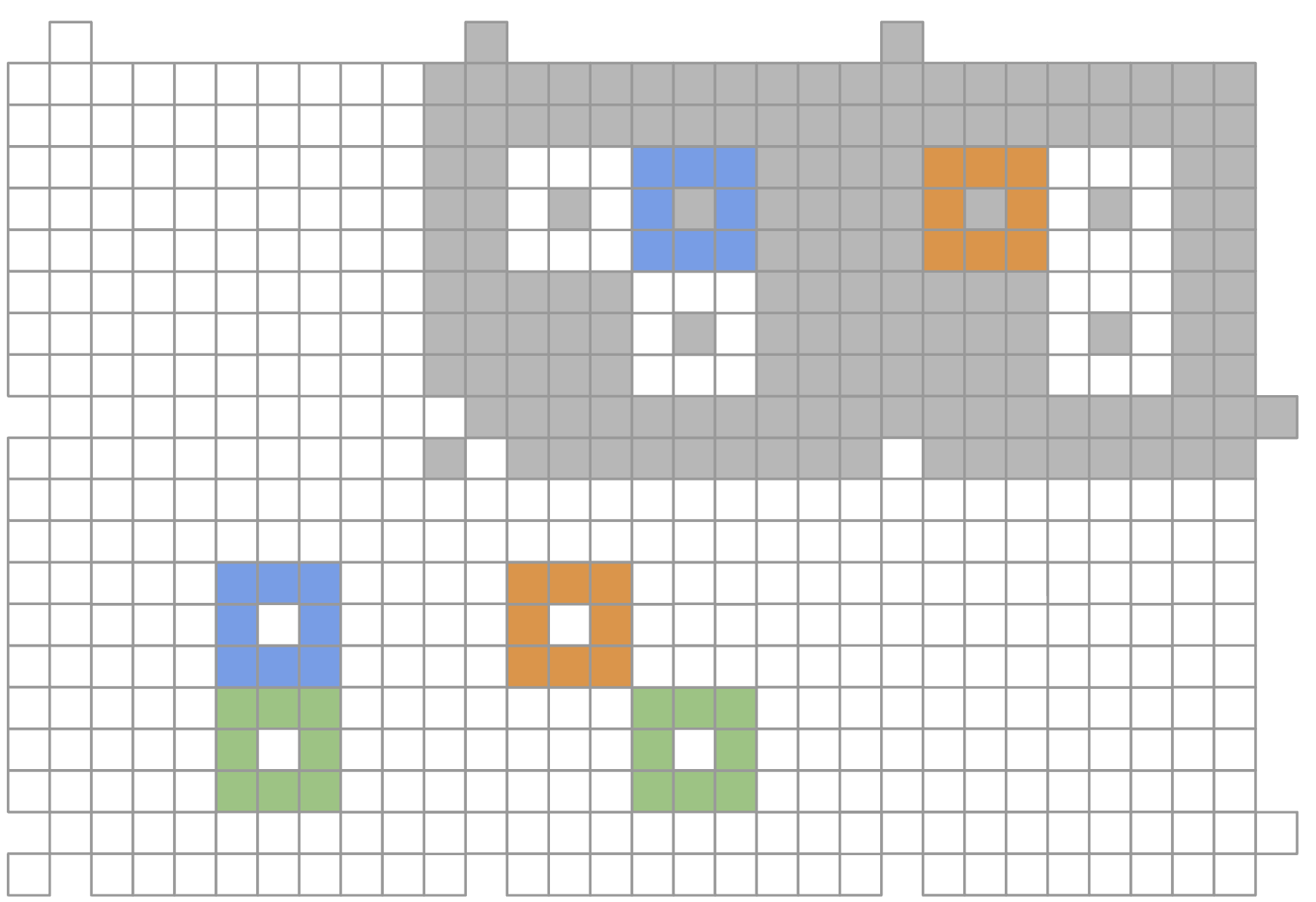}
    \caption{The tile $F(\Omega)\subset N^{-1}\Z^2$ for $\Omega$ as in Figure \ref{fig:lines}.  A point $(a,b)\in N^{-1}\Z^2$ is represented here by a square $(a,b)+[0,\tfrac{1}{N}]^2$. The colored squares correspond to the elements in $F\subset N^{-1} \mathbb Z^2$, while the white squares correspond to the elements in $N^{-1}\Z^2$ that are not in $F$.}
    \label{fig:new_tile}%
\end{figure}

We show that this $F$ ``encodes'' the tiling properties of $\Omega$:
\begin{lemma}\label{lem:raster}
    Let $\Omega$ be an integer set, and let $N=N(\Omega)$ and $F=F(\Omega)$ be as constructed above. Then we have:
    \begin{itemize}
        \item[(i)] The set $\Tile(\Omega;\R^2)$ is empty if and only if the set $\Tile(F;\tfrac{1}{N}\Z^2)$ is empty.
        \item[(ii)] For every set $0\in T\subset \R^2$,  $T\in\Tile(F; \tfrac{1}{N}\Z^2)$ if and only if  $T\in \Tile(\Omega;\R^2)$ with $T\subset \Z^2$.
        \item[(iii)] If $T_0, T_1$ are subsets of $\mathbb Z^2$, and $\vec v\in \Z^2$, then  $\Omega \oplus T_0 + \vec v = \Omega \oplus T_1$ if and only if  $F \oplus T_0 +\vec v = F \oplus T_1$.
    \end{itemize}
\end{lemma}
\begin{proof}
Observe that  (i) follows from  (ii)  and Lemma~\ref{lem:integertiling}. Indeed, if $\Tile(\Omega;\R^2)$  in non-empty then by Lemma~\ref{lem:integertiling} there exists $0\in T\in\Tile(\Omega;\R^2)$ such that $T\subset \Z^2$, thus, (ii) implies that $T\in \Tile(F;\tfrac{1}{N}\Z^2)$. Conversely, if  $\Tile(F;\tfrac{1}{N}\Z^2)$ is non-empty then it contains a set $ T$ such that $0\in T$; thus, again, by (ii), $T\in \Tile(\Omega;\R^2)$.

 We prove (ii). Suppose first that  $0\in T\in \Tile(\Omega;\R^2)$ and $T\subset \Z^2$. Then, for any $P_i\in\{P_0,\dots,P_{M-1}\}$ and $x\in\Z^2$ there exists a unique translate $t_{i,x}\in T$ so that $(x+P_i)  \subset  (\Omega+t_{i,x})$. Note  that by the definition of $S_i$ we have $(x+ P_i) \subset (\Omega+t_{i,x})$ if and only if $(x+ \frac 1 N  S_i) \subset (F + t_{i,x})$. This implies that
  $$ T\oplus F = \bigsqcup_{\substack{x\in \Z^2 \\ i\in\{0,\dots,M-1\}}} \left( x + \tfrac 1 N S_i\right) = \tfrac 1 N \Z^2 $$
  i.e., $T\in \Tile(F;\left( \frac1 N \Z\right)^2)$.
  
  Conversely, suppose that $0\in T\in \Tile(F;\frac 1 N \Z^2)$, then, by construction of $S_0$ and since $F$ is a tile of $\frac 1 N \Z^2$, we must have that the set $$F_0\coloneqq \bigsqcup_{v\in\Z^2}\bigsqcup_{(x+P_0) \subset \Omega}(x+\tfrac 1 N S_0) \subset F$$  is non-empty.
  Thus, since $T\in \Tile(F;\frac 1 N \Z^2)$, by construction of $F$, we have that  if $t,t'$ are two distinct elements of $T$ such that $F_0+[0,\frac 1 N]^2+t$ and $F_0+[0,\frac 1 N]^2+t'$ share an edge then $t'-t\in \Z^2$. This means that every set in $\Tile(F;\frac 1 N \Z^2)$ must be  contained in some coset of $\Z^2$ in $\frac 1 N \Z^2$. In particular, as by assumption $0\in T$, we must have $T\subset \Z^2$. 
  
   We now show that $T\in \Tile(\Omega; \R^2)$. Since $0\in T\in \Tile(F;\frac 1 N \Z^2)$,  for every $S_i\in\{S_0,\dots,S_{M-1}\}$ and $x\in\Z^2$,  there exists a unique $t_{i,x}\in T$ such that $(x+ \frac 1 N  S_i) \subset (F+t_{i,x})$. Thus, by construction of $S_i$, we have that
   $(x+ P_i) \subset (\Omega+t_{i,x})$; hence
  $$ T\oplus \Omega = \bigsqcup_{\substack{x\in \Z^2 \\ i\in\{0,\dots,M-1\}}} \left( x + P_i \right) = \R^2 $$
  i.e., $T\in \Tile(\Omega;\R^2)$. 

  It remains to prove (iii).  For $i=0,\dots,M-1$, let $$X_{i,0}\coloneqq \{x\in \Z^2\colon x+P_i\subset \Omega\oplus T_0 +\vec v\}, \quad X_{i,1}\coloneqq \{x\in \Z^2\colon x+P_i\subset \Omega\oplus T_1\},$$ and  $$Y_{i,0}\coloneqq \{x\in \Z^2\colon x+N^{-1}S_i\subset F\oplus T_0+\vec v\},\quad Y_{i,1}\coloneqq \{x+N^{-1}S_i\subset F\oplus T_1\}.$$ As argued above, we have $X_{i,j}=Y_{i,j}$ for every $i=1,\dots,M-1$ and $j=0,1$. 
  We therefore have that $X_{i,0}=X_{i,1}$ for every $i=0,\dots,M-1$ if and only if $Y_{i,0}=Y_{i,1}$ for every $i=0,\dots,M-1$. In other words, $\Omega \oplus T_0 + \vec v = \Omega \oplus T_1$ if and only if  $F \oplus T_0 + \vec v = F \oplus T_1$, as claimed.
\end{proof}

\subsection{} 

Combining Lemmas  \ref{lem:raster} with \cite[Theorem~1.5]{GT}, we can now prove Theorem~\ref{thm:ptc}, as follows.

\begin{proof}[Proof of Theorem~\ref{thm:ptc}]
    Let $\Omega$ be a rational polygonal tile and, by translation invariance, assume without loss of generality that $0\in V(\Omega)$. Then, as $V(\Omega)\subset \Q^2$ is finite, there exists $n\in \N$ such that $V(n\Omega)\subset \Z^2$. Let $\tilde \Omega \coloneqq n\Omega$. Note that $T\in \Tile(\Omega;\R^2)$ if and only if $nT \in \Tile(\tilde \Omega;\R^2)$ and $T$ is doubly-periodic if and only if $nT$ is doubly-periodic. 
Thus, we can equivalently prove the statement for $\tilde \Omega$. 

Let $N\in\N=N(\tilde \Omega)$ and  $F=F(\tilde \Omega)\subset N^{-1}\Z^2$ be as in  Lemma~\ref{lem:raster}.
Then, by Lemma~\ref{lem:raster}(i) and dilation invariance, as by assumption $\Tile(\tilde \Omega;\R^2)$ is non-empty, $\Tile(NF;\Z^2)$ is non-empty, i.e., $NF\subset \Z^2$ tiles $\Z^2$.

By \cite[Theorem~1.5]{GT}, we then have the existence of a doubly-periodic $\tilde T\subset \Z^2$  such that $\tilde T\in \Tile(NF;\Z^2)$. By Translation invariance, we can assume without loss of generality that $0\in\tilde T$. 

Letting $T\coloneqq N^{-1}\tilde T\subset N^{-1}\Z^2$ we have that $0\in T\in \Tile(F;N^{-1}\Z^2)$. By Lemma~\ref{lem:raster}(ii), we then have that $T\in\Tile(\tilde \Omega; \R^2)$ and $T\subset \Z^2$. As $T$ is doubly-periodic, this concludes the proof.
\end{proof}

\subsection{}
We now deduce Corollary~\ref{cor:decide}, as follows.

\begin{proof}[Proof of Corollary~\ref{cor:decide}]
    We describe an algorithm that computes whether a given rational polygonal set tiles $\R^2$ by translations:  Let $\Omega$ be a rational polygonal set. The algorithm first computes the number $N$ and the finite set $F\subset N^{-1}\Z^2$, as described in the proof of Lemma~\ref{lem:raster}. By Lemma~\ref{lem:raster}(i), $\Omega$ tiles $\R^2$ if and only if $F$ tiles $N^{-1}\Z^2$. 
    
    The algorithm then computes all the ways to tile finite regions in $N^{-1}\Z^2$ by translated copies of $F$ (``patches''),  with growing diameter. If $F$ doesn't tile, then, by compactness, the largest possible diameter of any patch formed by translated copies of
the $F$ must be finite. Thus the algorithm will detect that $F$, and thus $\Omega$, is not a tile after finitely many steps. If $F$ tiles, then a periodic tiling exists; and so, after finitely many steps the algorithm will detect a tiling of a patch that satisfies periodic boundary conditions (i.e., the patch a fundamental domain of some lattice in $N^{-1}\Z^2$), and will then output that $F$, and thus $\Omega$ is a tile.
\end{proof}

\section{Structure of tilings by a rational tile}\label{sec:weakper}

In this section, using Lemma~\ref{lem:raster} and analysing the structure of discrete tilings by $F(\Omega)$. and its connection to the structure of tilings by the original rational polygonal set, we prove Theorem~\ref{thm:weakper}. 

We begin by recalling the discrete analog of Theorem~\ref{thm:weakper}, which was established in \cite[Theorem~1.4]{GT}.

\begin{theorem}[Tilings in $\Z^2$ are weakly-periodic]\label{thm:weakZper}
    Let $F \subset \mathbb Z^2$ be a tile.
    Then there exist $c=c(F)$ and $d=d(F)$ such that there are $m < c$
    pairwise incommensurable vectors $h_1,\dots,h_m\in \mathbb{Z}^2\setminus \{0\}$ with $|h_j|\leq d^{m}$ such that every tiling $T$ of $\mathbb{Z}^2$ by $F$ admits a decomposition 
    \begin{equation}\label{eq:weakZper}
        T=T_1\sqcup \dots \sqcup T_m
    \end{equation} such that $T_j$ is $\langle h_j\rangle$-periodic. 
     In fact, there exists a lattice $\Lambda\subset \Z^2$ such that for every $x\in\Z^2$ there is $1\leq j\leq m$ such that the set  $(\Lambda+x)\cap T$ is $\langle h_j\rangle$-periodic.
  \end{theorem}

  \begin{proof}
      See \cite[Theorem~1.4]{GT}.
  \end{proof}

Observe that by combining Theorem~\ref{thm:weakZper} and Lemma~\ref{lem:raster}(ii), we can immediately deduce that any \emph{integer} tiling in $\Tile(\Omega;\R^2)$ is weakly periodic.
  The main barrier to transferring Theorem~\ref{thm:weakper} from its integer counterpart is the possible \emph{sliding} components obstructing periodicity. The proof of Theorem~\ref{thm:weakper} consists of showing that the decomposition of Theorem~\ref{thm:weakZper} is compatible with the sliding decomposition in Lemma~\ref{lem:tilings_slide}, in the sense that each of the singly-periodic pieces in the decomposition \eqref{eq:weakZper} can be further decomposed into pieces, each contained in a sliding component of the partition \eqref{eq:tiling_slide}. In order to show that, we introduce the notion \emph{earthquake\footnote{The terminology of \emph{earthquakes} also appears in \cite{ken}, but there it is used in the context of continuous tilings and the definition (although closely related) is slightly different.} decomposition} of discrete tilings.

\subsection{Earthquake decomposition}

\begin{definition}[Earthquake decomposition of a discrete tiling]
    Let $F\subset N^{-1}\Z^2$ be finite. Suppose $T \in \Tile(F; N^{-1}\Z^2)$, and $\vec v \in N^{-1}\Z^2$ then we define the  \emph{$\vec v$-discrete earthquake decomposition of $T$ by $F$} as a decomposition
   $$T = \bigsqcup_{P \in \quake{T}} P,$$
    where:
    \begin{enumerate}
        \item[(i)] For each set $P\in \quake{T}$ the set $F \oplus P$ is $\mathbb Z\vec v$-periodic.
        \item[(ii)] If $P\in \quake{T}$ and $P'\subset P$ satisfies (i) then $P'\in \{P,\emptyset\}$. 
    \end{enumerate}
    
     The set\footnote{Note that a decomposition satisfying (ii) is the maximal one amongst all decompositions satisfying (i) and therefore always exists and is unique.}  $\quake{T}$ is the set of \emph{tectonic plates} arising from earthquakes in the direction $\vec v$. 
\end{definition}

The following Lemma shows that the discrete decomposition into singly-periodic sets is compatible with a discrete earthquake decomposition.

\begin{lemma}\label{lem:EarthquakeStructure}
    Let $F\subset N^{-1}\Z^2$ be a tile of $N^{-1}\Z^2$,  $T\in\Tile(F;N^{-1}\Z^2)$ and $\quake{T}$ be the $\vec v$-earthquake decomposition of $T$.
    Then there are pairwise incommensurable vectors $h_1,\dots,h_m\in \mathbb{Z}^2\setminus \{0\}$ such that for each $P\in \quake{T}$ there is a decomposition $P=\bigsqcup_{j=1}^m P_j$ such that $P_j$ is $\langle h_j\rangle$-periodic for every  $1\le j\le m$. 
\end{lemma}

\begin{proof} Let $\tilde F\coloneqq NF$. Then, as $T\in\Tile(F;N^{-1}\Z^2)$ if and only if $NT\in \Tile(\tilde F;\Z^2)$, it suffices to prove the statement for any $T\in \Tile(\tilde F;\Z^2)$.
    By translating, we may assume without loss of generality that $\tilde F$ contains the origin.  
 Fix some $P\in \quake T$, and define $T^k \in \Tile(\tilde F,\mathbb Z^2)$ as $$T^k \coloneqq P\sqcup \bigsqcup_{P'\in \quake T \setminus \{P\}} P' + k \vec v.$$
    Note that $T^k$ is obtained from $T$ by an \emph{earthquake} that shifts all the plates but $P$ by $k \vec v$, and that $P \in \quake{T^k}$ for all $k$ by the construction of $T^k$.

By Theorem~\ref{thm:weakZper}, there is $M\in\N$ and pairwise incommensurable vectors $ \tilde h_1,\dots,\tilde h_{m}\in M\Z^2$ with the property that for every coset $\Gamma=(x+M\Z^2)\in \mathbb{Z}^2/M\Z^2$ there is $1\leq j=j(k,\Gamma) \leq m$ such that $\Gamma \cap T^k$ is $\langle \tilde  h_j\rangle$-periodic. By adding a vector (if needed) we can assume without loss of generality that $\tilde h_1$ is parallel to $\vec v$, and by enlarging $\tilde h_1$ (if needed) we can further assume that $\tilde h_1=\ell M\vec v$ for some $\ell\in\N$.
    
Let $\Gamma=(x+ M\Z^2)\in \Z^2/\Lambda$. For each $k \in \Z$, let $h(k,\Gamma)\coloneqq \tilde h_{j(k,\Gamma)}$. 
By the pigeonhole principle, there exist $0\le k<k'\le m$ such that $h(M k,\Gamma) = h(M k',\Gamma)$. We let $0\le k_\Gamma \le m-1$  be the smallest such that there exists $k_\Gamma < k' \le m$ with $h(M k_\Gamma,\Gamma) = h(M k',\Gamma)$ and set $k_\Gamma < k'_\Gamma \le m$ to be the smallest among all possible  $k'$ with this property. Let $h(\Gamma)\coloneqq h(Mk_\Gamma,\Gamma)$.

For each $j = 2, \dots m$, let $\mathcal{P}_j\subset P$ be the set of points $x$ in $P$ such that 
  $h(x+M\Z^2) = \tilde h_j$, and let $P_j\subset \mathcal{P}_j$ be the set of points $x$ in $\mathcal{P}_j$ such that 
  the coset $x + \tilde h_j \Z$ is entirely contained in $P$. 
   Clearly, for each $2\leq j\leq m$, the set $P_j$ is $\langle \tilde h_j\rangle$-periodic.

We will show that there exists $n\in\N$, independent of the choice of $P\in\quake{T}$, such that each of the sets $\tilde P_j\coloneq \mathcal{P}_j\setminus P_j$, $j=2,\dots,m$, is $\langle n\tilde h_1\rangle$-periodic. Then, the result will follow by setting $h_1\coloneqq n\tilde h_1$, $h_j\coloneqq \tilde h_j$ for $j=2,\dots,m$, and $$P_1 \coloneqq P\setminus \left(\sqcup_{j=2}^m P_j\right) =\{x\in P\colon h(x+M\Z^2)=\tilde h_1\} \cup \left( \sqcup_{j=2}^m \tilde P_j\right),$$ which is $\langle h_1\rangle$-periodic. 
It thus suffices to show  for each $j=2,\dots,m$  that the set  $\tilde P_j$ is $\langle n\ell M\vec v \rangle$-periodic with $n=\prod_{\Gamma\in\Z^2/M\Z^2}( k_\Gamma-k'_\Gamma)$.

Let $2\leq j\leq m$. Suppose that $x\in  \tilde P_j$ and let $\Gamma\coloneqq x+M\Z^2$. Then, by the definition of $\tilde P_j$,  there is $k\in\Z$ such that the point $y\coloneqq x+ k \tilde h_j\in \Gamma$ is in $T\setminus P$. Observe that as $\tilde h_j\in M\Z^2$ we have $\Gamma=x+M\Z^2=y+M\Z^2$. Thus, as by the definition of $(k_\Gamma,k_\Gamma')$ both $T^{Mk_\Gamma}\cap\Gamma$ and $T^{Mk_\Gamma'}\cap\Gamma$ are $\langle \tilde h_j\rangle$-periodic and $x\in T^{Mk_\Gamma}\cap T^{Mk_\Gamma'}\cap \Gamma$, we have that $$y=x+ k \tilde h_j \in T^{Mk_\Gamma}\cap T^{Mk_\Gamma'}\cap \Gamma.$$ Therefore, the point $z\coloneqq y-M k_{\Gamma}\vec v$ is in $T\cap \Gamma$. In turn, we obtain that  the point $$z'\coloneqq  z+Mk'_\Gamma\vec v= y-M(k_\Gamma-k_\Gamma')\vec v$$ is in $T^{Mk'_{\Gamma}}\cap \Gamma$; and so, by $\langle \tilde h_j\rangle$-periodicity of $T^{Mk'_{\Gamma}}\cap \Gamma$, we have that the point $$x'\coloneqq z'-k \tilde h_j = x-M(k_\Gamma-k_\Gamma')\vec v$$ is also in $T^{Mk'_{\Gamma}}\cap \Gamma$. 
Observe that since $x\in \tilde P_j\subset P$ and $P\in\quake{T}$, the point $x'$ is in $P$, and so, as $x'\in \Gamma$, we have that $x'\in\mathcal{P}_j$.
Moreover, since $y\in T\setminus P$ we have $y\not\in F+P$, and so, by the $\vec v$-periodicity of $F+P$, we obtain that neither $z$ nor $z'$  is in $F+P$ and, in particular, $z'=x'+k\tilde h_j \not\in P$. This shows that $x'=x-M(k_\Gamma-k'_\Gamma)\vec v$ is in $\tilde P_j$. 
Similarly, one can show that the point $x+M(k_\Gamma-k'_\Gamma)\vec v$ is in $\tilde P_j$, by simply repeating the above argument with the roles of $k_\Gamma$ and $k_\Gamma'$ being swapped. 
We therefore conclude that $\tilde P_j$ is $\langle n\tilde h_1\rangle$-periodic with $n=\prod_{\Gamma\in\Z^2/M\Z^2}(k_\Gamma-k'_\Gamma)$. 
\end{proof}

\subsection{}
Observe that using Lemma~\ref{lem:EarthquakeStructure}, to complete the proof of Theorem~\ref{thm:weakper}, it only remains to show that the discrete earthquake decomposition in Lemma~\ref{lem:EarthquakeStructure} is compatible with the continuous earthquake structure in $\Omega$.

\begin{proof}[Proof of Theorem~\ref{thm:weakper}]
Let $\Omega$ be a rational polygonal tile, by translation invariance, we can assume without loss of generality that $0\in V(\Omega)$. 
Let $T\in\Tile(\Omega;\R^2)$. By Lemma~\ref{lem:tilings_slide}, there exists a direction $\vec v\in \mathbb R^2\setminus \{0\}$ such that  $T$ can be decomposed as $T = \bigsqcup_{i\in I} Q_i$, where for each $i\in I$ the set $\tilde \Omega\oplus Q_i$ is $\vec v\R$-invariant.  By invariance under affine transformations, we may assume that $\vec v=(0,1)$.

As $V(\Omega)\subset \Q^2$ is finite, there exists $n\in \N$ such that $V(n\Omega)\subset \Z^2$. Let $\tilde \Omega \coloneqq n\Omega$. Note that $T\in \Tile(\Omega;\R^2)$ if and only if $nT \in \Tile(\tilde \Omega;\R^2)$ and $T$ is weakly-periodic if and only if $nT$ is weakly-periodic. 
Thus, we can equivalently prove the statement for $\tilde \Omega$.

Let $T'\in\Tile(\tilde \Omega; \R^2)$ be the tiling that arises from merging the $\sim_{\tilde \Omega}$ components of $T$ by sliding, as described in Lemma~\ref{lem:onecomp}.
Then, by construction, there are coefficients $s_i\in \mathbb{R}$ such that $$ T' = \bigsqcup_{i\in I} Q_i - s_i (0,1),$$ 
and by Lemma~\ref{lem:tile+vertex=>tiling} $T'$ is a subset of $\Z^2$, i.e., $T$ is an integer tiling by $\tilde \Omega$.

Let $F=F(\tilde \Omega)$ be as in \eqref{eq:F}. Then, by Lemma~\ref{lem:raster}(ii) $ T'\in \Tile(F;N^{-1}\Z^2)$, and by Lemma~\ref{lem:raster}(iii), if for a subset $S$ of $T'$ the set $\tilde \Omega \oplus S$ is $(0,1)\mathbb R$-invariant, then $F\oplus S$ is $(0,1)\Z$-invariant. Therefore, the decomposition of $T'$ into $(0,1)$-discrete earthquakes by $F$ is a refinement of the decomposition of $T'$ into continuous earthquakes by $\tilde \Omega$. In other words, for $i\in I$ there exist sets $P_{i,j}\subset \Z^2$, $j\in J_i$ such that $Q_i - s_i(0,1)= \bigsqcup_{j\in J_i}  P_{i,j} $ and the $(0,1)$-earthquake decomposition of $T'$ by $F$ is $ T' = \bigsqcup_{i,j}  P_{i,j}$.

 Thus, by Lemma~\ref{lem:EarthquakeStructure}, for each $i\in I$, $j\in J_i$ there exists $h_{i,j}\in \Z^2$ such that the set $P_{i,j}$ is $\langle h_{i,j}\rangle$-periodic, and the vectors $h_{i,j}$ take finitely many possible values in $\Z^2$. Therefore, each $Q_i-s_i(0,1)=\bigsqcup_{j\in J_i}  P_{i,j}$ is weakly-periodic. By translation invariance, this implies that each $Q_i$ is weakly-periodic with the same periods. Therefore, as $T = \bigsqcup_{i\in I} Q_i$, the tiling $T$ is weakly-periodic. We thus have Theorem~\ref{thm:weakper}.
\end{proof}

\end{document}